\documentclass[11pt]{article}

\usepackage[
	a4paper,
	margin=1in
]{geometry}
\usepackage{setspace}
\usepackage{sectsty} % for \allsectionsfont
\usepackage{titlesec}
\usepackage{appendix}

\usepackage{amsmath, amssymb, amsthm}
\usepackage{mathtools}

\usepackage[
	setpagesize=false,
	colorlinks=true,
	linkcolor=BrickRed,
        citecolor=OliveGreen,
        urlcolor=black,
	pdfencoding=auto,
	psdextra,
]{hyperref}
\usepackage{cleveref}

\usepackage{etoolbox} % for \AtBeginEnvironment
\usepackage[shortlabels]{enumitem}
\usepackage{xparse} % for \mins

\usepackage{graphicx}
\usepackage[dvipsnames]{xcolor}
\usepackage{tikz}

\usepackage{todonotes}
\usepackage[
    deletedmarkup=sout,
    commentmarkup=uwave,
    authormarkuptext=name
]{changes}

\usepackage{comment}

\setlist[enumerate,1]{label=(\arabic*), ref=(\arabic*)}
\setlist[enumerate,3]{label=(\roman*), ref=(\roman*)}

\allsectionsfont{\boldmath} % embolden math in title

\titlelabel{\thetitle.\quad}

\theoremstyle{plain}
\newtheorem{theorem}{Theorem}[section]
\newtheorem{lemma}[theorem]{Lemma}

\newtheorem{conjecture}[theorem]{Conjecture}

\newtheorem{question}[theorem]{Question}

\newtheorem{claim}[theorem]{Claim}
\newtheorem*{claim*}{Claim}

\makeatletter
\newenvironment{claimproof}[1][Proof]{\par
	\pushQED{\qed}%
	
	\normalfont \topsep6\p@\@plus6\p@\relax
	\trivlist
	\item[\hskip\labelsep
	\textit{#1}\@addpunct{.}~]\ignorespaces
}{%
	\popQED\endtrivlist\@endpefalse
}
\makeatother

\theoremstyle{definition}

\newtheorem*{definition*}{Definition}
\newcommand{\calF}{\mathcal{F}}

\newcommand{\calT}{\mathcal{T}}

\newcommand{\ve}{\varepsilon}

\newcommand{\defeq}{\coloneqq}

\title{Anticoncentration of random spanning trees in almost regular graphs}

\author{
Hyunwoo Lee%
        \thanks{Department of Mathematical Sciences, KAIST, South Korea and Extremal Combinatorics and Probability Group (ECOPRO), Institute for Basic Science (IBS).
        E-mail: {\ttfamily hyunwoo.lee@kaist.ac.kr.} Supported by the National Research Foundation of Korea (NRF) grant funded by the Korea government(MSIT) No. RS-2023-00210430, and the Institute for Basic Science (IBS-R029-C4).}
}

\usepackage[square,sort,comma,numbers]{natbib}
\begin{document}
\maketitle

\begin{abstract}
    The celebrated formula of Otter \emph{[Ann. of Math. (2) 49 (1948), 583--599]} asserts that the complete graph contains exponentially many non-isomorphic spanning trees.
    In this paper, we show that every connected almost regular graph with sufficiently large degree already contains exponentially many non-isomorphic spanning trees.
    Indeed, we prove a stronger statement: for every fixed $n$-vertex tree $T$,
    $$
        \Pr\bigl[\mathcal{T} \simeq_{\mathrm{iso}} T\bigr] = e^{-\Omega(n)},
    $$
    where $\mathcal{T}$ is a uniformly random spanning tree of a connected $n$-vertex almost regular graph with sufficiently large degree.
    To prove this, we introduce a graph-theoretic variant of the classical balls--into--bins model, which may be of independent interest.
\end{abstract}

%---------------------------------------------------------

\section{Introduction}\label{sec:intro}

Counting the number of spanning trees of a given graph is a fundamental problem in combinatorics, with deep connections to probability theory and statistical physics, and has been extensively studied over the past century.
For a graph $G$, let $T(G)$ denote the set of its spanning trees.
Two classical landmark results in graph theory concerning the quantity $|T(G)|$ are Cayley’s formula~\cite{Cayley} and Kirchhoff's Matrix--Tree Theorem~\cite{Kirchhoff}.
Cayley’s formula determines the exact number of spanning trees of the complete graph on $n$ vertices, asserting that $|T(K_n)| = n^{n-2}$.
Beyond complete graphs, the Matrix--Tree Theorem provides a general method to compute $|T(G)|$ exactly via the eigenvalues of the Laplacian matrix of $G$.

There has been extensive research aimed at extending Cayley’s formula and the Matrix--Tree Theorem to derive bounds on $|T(G)|$ for various graphs $G$, such as being bipartite graphs~\cite{Fiedler-Sedlavek}, graphs with prescribed degree sequences~\cite{Kostochka}, $m$-edge graphs~\cite{Grimmett,Das}, and regular graphs~\cite{McKay,Alon}.
In particular, the quantity $|T(G)|$ has attracted significant attention in the case where $G$ is a connected regular graph.
Motivated by Cayley’s formula, McKay~\cite{McKay} showed that
$$
    \limsup_{n\to \infty} \{|T(G)|^{1/n}: \text{$G$ is a connected $n$-vertex $d$-regular graph}\} = \frac{(d-1)^{d-1}}{(d^2-2d)^{d/2 - 1}}, 
$$ where $d \geq 3$.
On the other hand, concerning lower bounds, Alon~\cite{Alon} proved that
$$
    |T(G)| \geq d^{(1 - o_d(1))n}
$$
for every connected $n$-vertex $d$-regular graph $G$.
Consequently, these results yield fairly precise asymptotic estimates for the number of spanning trees in connected regular graphs, summarized as follows.

\begin{theorem}[McKay~\cite{McKay}, Alon~\cite{Alon}]\label{thm:labeled-tree}
    Let $G$ be a connected $n$-vertex $d$-regular graph.
    Then
    $$
        |T(G)| = d^{(1 - o_d(1))n}.
    $$
\end{theorem}

By extending Alon's result, Kostochka~\cite{Kostochka} established the following asymptotically tight inequality that can be applied to many graphs,

\begin{theorem}[Kostochka~\cite{Kostochka}]\label{thm:Kostochka}
    Let $G$ be a connected $n$-vertex graph with minimum degree $k$. Let $d_g(G)$ be the geometric average degree of $G$, say $d_g(G):= \left( \prod_{v\in V(G)} d_G(v)\right)^{1/n}$. Then the following inequality holds.
    $$
        d_g(G)^{n}e^{-O(\log^2 k/k)n} \leq |T(G)| \leq d_g(G)^n/(n-1).
    $$
\end{theorem}

So far, we have focused on the number of labeled spanning trees contained in a given graph. A natural next question is to consider the number of unlabeled $n$-vertex trees that are contained in an $n$-vertex graph $G$. Equivalently, the number of isomorphism classes of spanning trees in $G$. We denote $T_{\mathrm{unlabeled}}(G)$ by the set of unlabeled spanning trees that are contained in $G$.
As an unlabeled version of Cayley's formula, a remarkable result of Otter~\cite{Otter} says that the number of unlabeled trees on $n$ vertices is $(1 + o(1))C n^{-5/2} \alpha^n$, where $C \approx 0.535$ and $\alpha \approx 2.956$. In other words, the following is the Otter's formula.

\begin{theorem}[Otter~\cite{Otter}]\label{thm:otter}
    $$
        |T_{\mathrm{unlabeled}}(K_n)| = (1 + o(1))C n^{-5/2} \alpha^n,
    $$ where $C \approx 0.535$ and $\alpha \approx 2.956$.
\end{theorem}

Note that Theorem~\ref{thm:otter} asserts that there are exponentially many pairwise non-isomorphic spanning trees on $n$ vertices. Regarding non-isomorphic spanning trees, several studies were conducted in the 1970s and 1980s (see~\cite{Vestergaard}); however, in contrast to the extensive literature on counting labeled spanning trees, there are relatively few works addressing the enumeration of pairwise non-isomorphic spanning trees in a given graph. One reason for this scarcity is the wide variety of tree structures and the lack of effective methods to control the sizes of their automorphism groups, which makes such counts inherently difficult. In this context, somewhat surprisingly, we show that every connected almost-regular graph with degrees bigger than some universal constant already contains exponentially many pairwise non-isomorphic spanning trees.
For constants $d, \delta > 0$, we say a graph $G$ is \emph{$(1 \pm \delta)d$-regular}, if all the degrees of $G$ are lies on the interval $[(1 - \delta)d, (1 + \delta)d]$.

\begin{theorem}\label{thm:main-non-isomorphic}
    There exists a universal constant $d_0> 0$ such that the following holds for all integers $d \geq d_0$.
    Let $G$ be a connected $n$-vertex $(1\pm 10^{-10})d$-regular graph. Then we have
    $$
        |T_{\mathrm{unlabeled}}(G)| \geq e^{n/2000}.
    $$
\end{theorem}

We prove Theorem~\ref{thm:main-non-isomorphic} by establishing a stronger statement, Theorem~\ref{thm:main-anticoncentration}: in any connected almost-regular graph $G$, the largest collection of pairwise isomorphic spanning trees is exponentially smaller than $|T(G)|$. Probabilistically, this yields a strong anticoncentration property for random spanning trees relative to the set of unlabeled $n$-vertex trees, which is of independent interest.

\begin{theorem}\label{thm:main-anticoncentration}
    There exists a universal constant $d_0> 0$ such that the following holds for all integers $d \geq d_0$. Let $G$ be a connected $n$-vertex $(1\pm 10^{-10})d$-regular graph and $\calT$ be a random spanning tree chosen from the set of all spanning trees in $G$ uniformly at random.
    Then for every $n$-vertex tree $T$, we have
    $$
        \Pr[\calT \simeq_{\mathrm{iso}} T] \leq e^{-n/2000}.
    $$
\end{theorem}

The concept of anticoncentration has been extensively studied over the past few decades and has found numerous applications in probability theory, random matrix theory, and statistics. Beyond the classical Littlewood--Offord–type problems~\cite{Littlewood-Offord,Littlewood-Offord-Erdos} concerning sums of independent random variables, recent work has investigated anticoncentration phenomena for a variety of mathematical objects. For example, edge-statistics in graphs~\cite{Alon-Hefetz-Krivelevich-Tyomkyn,Fox-Kwan-Sauermann,Kwan-Sah-Sauermann-Sawhney}, subgraph counts~\cite{Meka-Nguyen-Vu,Sah-Sawheny}, permutation variants of the Littlewood--Offord problem~\cite{Pawlowski,Do-Hoi-Kiet-Tuan-Vu}. In this context, Theorem~\ref{thm:main-anticoncentration} reveals a new anticoncentration phenomenon for spanning trees.

Most existing proofs concerning the enumeration of spanning trees proceed by analyzing the spectrum of the Laplacian matrix via the Matrix--Tree Theorem, by employing generating functions, or by directly exploiting the structural properties of the underlying graph. In contrast, our proofs of Theorems~\ref{thm:main-non-isomorphic} and~\ref{thm:main-anticoncentration} are based on probabilistic techniques. Specifically, we introduce a graph-theoretic variant of the classical balls--into--bins model (see Section~\ref{sec:balls-into-bins}) and analyze it using several concentration inequalities such as Talagrand’s inequality. Also, we use well-known Br\'{e}gman--Minc inequality~\cite{Bregman}. We expect that these methods will find further applications in related problems.

Finally, we note that the conclusions of Theorems~\ref{thm:main-non-isomorphic} and \ref{thm:main-anticoncentration} fail if a minimum degree condition replaces the almost-regularity assumption. For example, the number of unlabeled spanning trees of the complete bipartite graph $K_{d,n-d}$ grows only polynomially in $n$, with degree and coefficients depending solely on $d$.

%-------------------------------------------------------

\section{Balls--into--bins on almost regular graphs}\label{sec:balls-into-bins}

To prove Theorem~\ref{thm:main-anticoncentration}, we adopt a graph-theoretic interpretation of the classical probabilistic model known as the balls--into--bins model. We refer the reader to Section~\ref{sec:proof} for a detailed explanation of how this model is employed in our setting.

Although numerous variants of the balls--into--bins model have been studied, the most basic and well-understood version is defined as follows. Suppose there are $m$ balls and $n$ bins. Each ball independently selects one bin uniformly at random from the $n$ bins. This process has been extensively analyzed in the literature, with particular attention to quantities such as the maximum load among the bins, the number of balls required to ensure that every bin receives at least one ball, and the asymptotic distribution of the number of bins containing a given number of balls. These analyses have found many applications in theoretical computer science, notably in problems related to load balancing and hashing. (See~\cite{Balls-into-bins-1,Balls-into-bins-2,Balls-into-bins-3}.)

Observe that one can relate the balls-into-bins model to a probabilistic model on a complete bipartite graph. Let $H$ be a bipartite graph on a bipartition $X \cup Y$, where $|X| = m$ and $|Y| = n$. Consider a random subgraph $H'$ on $H$ that is constructed by choosing a vertex $y\in Y$ from the neighborhood of a vertex $x\in X$ uniformly at random and putting $xy$ as an edge of $H'$, where all choices are made independently. If $H$ is a complete bipartite graph, then analyzing the degrees of vertices of $Y$ in $H'$ is equivalent to the classical balls-into-bins model with $m$ balls and $n$ bins. We say this graph analog of the balls--into--bins model as \emph{$(X, Y; H)$-model} and the random subgraph $H'$ that is generated by the demonstrated process as the \emph{output graph} of $(X, Y; H)$-model.

In this section, we consider the case where the bipartite graph $H$ is almost regular. This setting is of independent interest and will also play an important role in the proofs of our main results, Theorems~\ref{thm:main-non-isomorphic} and \ref{thm:main-anticoncentration}.

\begin{theorem}\label{thm:balls-into-bins-regular}
    For every integer $k \geq 0$ and real number $\delta \in \left[0, \frac{10^{-8}}{(k+2)!^2}\right]$, there exists a constant $d_0 > 0$ such that the following holds for all integers $d \geq d_0$. Consider the $(X, Y; H)$-model with $|X| = |Y| = n$, where $H$ is a $(1\pm \delta)d$-regular bipartite graph. Let $H'$ be the output graph of the $(X, Y; H)$-model, and let $I_k$ be the random variable counting the number of vertices in $Y$ of degree exactly $k$ in $H'$. Then, for every $\gamma \geq 600(k+1)\sqrt{\delta}$,
    \begin{equation}\label{eq:exponential-decay}
        \Pr\left[\left| I_k - \frac{e^{-1}}{k!}n \right| > \gamma n \right] \leq 10e^{-\frac{(\gamma - 600(k+1)\sqrt{\delta})^2}{32(k+1)}n}.
    \end{equation}
\end{theorem}

We remark that Theorem~\ref{thm:balls-into-bins-regular} shows that when $H$ is almost regular, the degree distribution of the output graph of the $(X, Y; H)$-model converges sharply to the Poisson distribution with mean $1$. In particular, the deviation from the Poisson expectation exhibits exponential decay.

To establish the exponential tail bound in \eqref{eq:exponential-decay}, we apply Talagrand’s concentration inequality~\cite[Section 7.7]{Alon-Spencer}. For each $i\in [n]$, let $\Omega_i$ be a probability space and $\Omega := \prod_{i\in [n]} \Omega_i$ has the product measure. We call $h: \Omega \to \mathbb{R}$ a \emph{$K$-Lipschitz} function if $|h(\mathbf{x}) - h(\mathbf{y})| \leq K$ whenever $\mathbf{x}, \mathbf{y}$ differ in at most one coordinate. We say $h$ is \emph{$f$-certificate} for a function $f:\mathbb{N}\to \mathbb{N}$ if whenever $h(\mathbf{x}) \geq \ell$, there exists a coordinate set $J\subseteq [n]$ with $|J| \leq f(\ell)$ such that all $\mathbf{y}\in \Omega$ that agree with $\mathbf{x}$ on the coordinates in $J$ have $h(\mathbf{y}) \geq \ell$. Then the following is Talagrand's concentration inequality.

\begin{lemma}[Talagrand]\label{lem:Talagrand}
    Let $\Omega_{i}$ be a probability space for each $i\in [n]$ and $\Omega := \prod_{i\in [n]} \Omega_i$ has the product measure. Let $h: \Omega \to \mathbb{R}$ be a $K$-Lipschitz and $f$-certificate function. Let $\mathbf{x}$ be a random point chosen from $\Omega$ according to the measure on $\Omega$. Then for any $b$ and  $t \geq 0$, the following holds.
    $$
        \Pr[h(\mathbf{x}) \leq b - tK\sqrt{f(b)}]\Pr[h(\mathbf{x}) \geq b] \leq e^{-t^2/4}.
    $$
\end{lemma}

Our strategy is to prove Theorem~\ref{thm:balls-into-bins-regular} by using Lemma~\ref{lem:Talagrand}, but to do this, we need to check two conditions, the Lipschitz and certificate conditions. For the random variable $I_k$, it is easy to check that $I_k$ is a $1$-Lipschitz function, but the problem arises in obtaining a sensible certificate function for $I_k$. To overcome this difficulty, we newly introduce a new random variable $I_{>k}$ which has a monotone property in some sense.

Let $X, Y,$ and $H$ be as in the statement of Theorem~\ref{thm:balls-into-bins-regular}, and write $X = \{x_1,\dots,x_n\}$.
For each $i \in [n]$, let $\Omega_i$ denote the probability space obtained by endowing the neighborhood $\Omega_i := N_H(x_i)$ with the uniform probability measure. We then define $\Omega := \prod_{i \in [n]} \Omega_i$ and equip $\Omega$ with the corresponding product measure. Note that each $\mathbf{y} \in \Omega$ corresponds to an output graph, denoted by $H'_{\mathbf{y}}$, of the $(X, Y; H)$-model with the same distribution. For an integer $m \geq 0$, we define the random variables $I_{> m}$, $I_m$ such that $I_{>m}$ ($I_m$) is the number of vertices of $Y$ that have degree larger than (exactly) $m$ in $H'_{\mathbf{y}}$. We remark that $I_m = I_{> m-1} - I_{> m}$ for each $m\geq 1$ and $I_0 = n - I_{ > 0}$.

As Lemma~\ref{lem:Talagrand} is a concentration inequality on the median, we need to estimate the median of $I_{>k}$. To do this, we will show that each $I_{m}$ is concentrated on the expectation for each $0 \leq k \leq m$ by using Chebyshev's inequality. We first obtain upper and lower bounds for the expectation.

\begin{lemma}\label{lem:expectation}
    Let $k, \delta, d_0, d, X, Y, H$ be the real numbers, vertex sets, and the bipartite graph as in the statement of Theorem~\ref{thm:balls-into-bins-regular}. Then for each $0 \leq m \leq k$, we have
    $$
        \left| \mathbb{E}[I_{m}] - \frac{e^{-1}}{m!}n \right| \leq 100(k+1)\delta \cdot \frac{e^{-1}}{m!}n.
    $$
\end{lemma}

\begin{proof}[Proof of Lemma~\ref{lem:expectation}]
    For each vertex $y\in Y$, denote $p^{(m)}_y$ by the probability that the degree of $y$ in the output graph $H'$ of the $(X, Y; H)$-model is exactly $m$. Then by the linearity of expectation, we have
    \begin{equation}\label{eq:deduct-py}
        \mathbb{E}[I_{m}] = \sum_{y\in Y} p^{(m)}_y.
    \end{equation}

    Let $v_1, \dots, v_s$ be the neighbors of $y$ in $H$. Then by the independence, the following holds.
    \begin{equation}\label{eq:py}
        p^{(m)}_y =  \sum_{J\subseteq [s], |J| = m} \prod_{j\in J}\frac{1}{d_H(v_j)} \prod_{j' \in [s]\setminus J} \frac{d_H(v_{j'}) - 1}{d_H(v_{j'})}.
    \end{equation}

    Since $H$ is $(1 \pm \delta)d$-regular, we have
    \allowdisplaybreaks
    \begin{align}
        \prod_{j\in J}\frac{1}{d_H(v_j)} \prod_{j' \in [s]\setminus J} \frac{d_H(v_{j'}) - 1}{d_H(v_{j'})} &\leq \frac{1}{(1 - \delta)^m d^m} \left(1 - \frac{1}{(1 + \delta)d}\right)^{(1-\delta)d-m}\nonumber \\
        %---------------
        &\leq \frac{(d - m)!}{(1 - \delta)^m d!}\left(1 - \frac{1}{(1 + \delta)d}\right)^{(1-2\delta)d}\nonumber\\
        %---------------
        &= \frac{(d - m)!}{(1 - \delta)^m d!}\left(1 - \frac{1}{(1 + \delta)d}\right)^{(1+\delta)d} \left(1 - \frac{1}{(1+\delta)d}\right)^{-3\delta d}\nonumber\\
        %---------------
        &\leq \frac{(d - m)!}{(1 - \delta)^m d!}\left(1 - \frac{1}{(1 + \delta)d}\right)^{(1+\delta)d} e^{\frac{3\delta}{1 + \delta}}\label{eq:1+xex}\\
        %---------------
        &\leq (1 + 6\delta)\frac{(d - m)!}{(1 - \delta)^m d!}\left(1 - \frac{1}{(1 + \delta)d}\right)^{(1+\delta)d}\label{eq:ex1+2x}\\
        %---------------
        &\leq (1 + 10\delta) \frac{(d - m)!e^{-1}}{(1 - \delta)^m d!}\label{eq:converge-to-e}\\
        %---------------
        &\leq (1 + 10\delta) (1 + 2\delta)^m \frac{(d - m)!e^{-1}}{d!}\nonumber \\
        %---------------
        &\leq (1 + \delta)^{2m+10} \frac{(d - m)!e^{-1}}{d!}.\nonumber
    \end{align}

    The inequality \eqref{eq:1+xex} holds due to the fact that $1 + x \leq e^x$. Also \eqref{eq:ex1+2x} holds because of the inequality $e^x \leq 1 + 2x$ holds whenever $x\in [0, 1/2]$. Finally, \eqref{eq:converge-to-e} follows from the limit $\lim_{x\to \infty} (1-x^{-1})^x = e^{-1}$ together with the assumption that $d$ is sufficiently large compared with $k$ and $\delta$. A matching lower bound can be obtained by analogous computations, and we omit the details. The following is what we have at this stage.

    \begin{equation}\label{eq:right-term}
        \prod_{j\in J}\frac{1}{d_H(v_j)} \prod_{j' \in [s]\setminus J} \frac{d_H(v_{j'}) - 1}{d_H(v_{j'})} = (1 \pm \delta)^{2m+10} \frac{(d - m)!e^{-1}}{d!},
    \end{equation}
    where $J\subseteq [s]$ and $|J| = m$.

    We now estimate the value $\binom{s}{m}\frac{(d-m)!}{d!}$. Since $(1 - \delta)d \leq s \leq (1 + \delta)d$, we have the follwoing.

    \allowdisplaybreaks
    \begin{equation*}
        \binom{s}{m}\frac{(d-m)!}{d!} = \frac{s(s-1)\cdots (s-m+1)}{d(d-1) \cdots (d-m+1)} \frac{1}{m!} \leq (1 + 2\delta)^m \frac{1}{m!} \leq (1 + \delta)^{2m}\frac{1}{m!}.
    \end{equation*}
    Similarly, we have
    \begin{equation*}
        \binom{s}{m}\frac{(d-m)!}{d!} = \frac{s(s-1)\cdots (s-m+1)}{d(d-1) \cdots (d-m+1)} \frac{1}{m!} \geq (1 - 3\delta/2)^m \frac{1}{m!} \geq (1 - \delta)^{2m}\frac{1}{m!}.
    \end{equation*}        

    Hence, we obtain
    \begin{equation}\label{eq:middle}
        \binom{s}{m}\frac{(d-m)!}{d!} \leq (1 \pm \delta)^{2m}\frac{1}{m!}.
    \end{equation}

    By combining \eqref{eq:py}, \eqref{eq:right-term}, \eqref{eq:middle}, we have

    \begin{equation*}
        \left| \mathbb{E}[I_{m}] - \frac{e^{-1}}{m!}n \right| \leq 100(k+1)\delta \cdot \frac{e^{-1}}{m!}n.
    \end{equation*}
    This completes the proof.
\end{proof}

The next lemma states that $I_m$ has a small variance.

\begin{lemma}\label{lem:variance}
    Let $k, \delta, d_0, d, X, Y, H$ be the real numbers, vertex sets, and the bipartite graph as in the statement of Theorem~\ref{thm:balls-into-bins-regular}. Then for each $0 \leq m \leq k$, we have
    $$
       \sqrt{\mathrm{Var}[I_m]} \leq 100\sqrt{(k+1)\delta} \cdot \frac{e^{-1}}{m!}n.
    $$
\end{lemma}

\begin{proof}[Proof of Lemma~\ref{lem:variance}]
    Since $\mathrm{Var}[I_m] = \mathbb{E}[I_m^2] - \mathbb{E}[I_m]^2$, we will estimate the value $\mathbb{E}[I_m^2]$. For a vertex pair $(y_1, y_2)\in Y^2$, denote $p^{(m)}_{(y_1, y_2)}$ by the probability that the degree of both $y_1$ and $y_2$ in the output graph $H'$ of the $(X, Y; H)$-model are exactly $m$. Then by the linearity of the expectation, we have
    
    \begin{equation}\label{eq:y1y2-linearlity}
        \mathbb{E}[I_m^2] = \sum_{(y_1, y_2)\in Y^2} p^{(m)}_{(y_1, y_2)}.
    \end{equation}

    For a given $(y_1, y_2)\in Y^2$, let $X_1 \defeq N_H(y_1) \setminus N_H(y_2)$, $X_2 \defeq N_H(y_2) \setminus N_H(y_1)$, and $Z\defeq N_H(y_1) \cap N_H(y_2)$.
    For a given vertex sets $W\subseteq X_1 \cup X_2$ and $W' \subseteq Z$, define $q(W, W')$ as 
    $$
        q(W, W') \defeq \prod_{v\in W\cup W'} \frac{1}{d_H(v)} \prod_{u\in (X_1 \cup X_2)\setminus{W}} \left(1 - \frac{1}{d_H(u)} \right) \prod_{w\in Z\setminus{W'}} \left(1 - \frac{2}{d_H(w)} \right).
    $$

    Then we have the following.
    
    \begin{equation}\label{eq:y1y2-prob}
        p^{(m)}_{(y_1, y_2)} = \sum_{a_1 + b_1 = m,\atop a_2 + b_2 = m} \sum_{V_1 \subseteq X_1, \atop |V_1| = a_1}\sum_{V_2 \subseteq X_2, \atop |V_2| = a_2}\sum_{U_1 \subseteq Z, \atop |U_1| = b_1}\sum_{U_2 \subseteq Z\setminus U_1, \atop |U_2| = b_2} q(V_1 \cup V_2, U_1 \cup U_2).
    \end{equation}

    To obtain a upper bound of $P^{(m)}_{(y, y')}$, we now bound the value $q(W, W')$ for all $W$ and $W'$ with $|W| + |W'| = 2m$. Since $H$ is $(1 \pm \delta)d$-regular graph, by the definition of $q(W, W')$, we have

    \allowdisplaybreaks
    \begin{align*}
        q(W, W') &\leq \left( \frac{1}{(1 - \delta)d} \right)^{|W| + |W'|} \prod_{u\in (X_1 \cup X_2)\setminus{W}} e^{-\frac{1}{d_H(u)}} \prod_{w\in Z\setminus{W'}} e^{-\frac{2}{d_H(w)}}\\
        %--------------------
        &\leq \left( \frac{1}{(1 - \delta)d} \right)^{|W| + |W'|} \prod_{u\in (X_1 \cup X_2)\setminus{W}} e^{-\frac{1}{(1 + \delta)d}} \prod_{w\in Z\setminus{W'}} e^{-\frac{2}{(1 + \delta)d}}\\
        %---------------------
        &\leq ((1 - \delta)d)^{-(|W| + |W'|)} e^{- \left(\frac{|X_1| + |X_2| - |W| + 2|Z| - 2|W'|}{(1 + \delta)d}\right)}\\ 
        %---------------------
        &\leq ((1 - \delta)d)^{-(|W| + |W'|)} e^{\left(-\frac{2(1-\delta)}{(1+\delta)} + \frac{2(|W| + |W'|)}{d}\right)}\\
        &\leq ((1 - \delta)d)^{-(|W| + |W'|)} e^{-2 + 4\delta + \frac{2(|W| + |W'|)}{d}}\\
        %---------------------
        &= (1 - \delta)^{-2m}d^{-2m}e^{-2 + 4\delta + 4m/d}\\
        %---------------------
        &\leq (1 - \delta)^{-2m}\left(\frac{(d-m)!}{d!}\right)^2e^{-2 + 10\delta}.
    \end{align*}

    Hence, from \eqref{eq:y1y2-prob}, the following inequality holds.
    \allowdisplaybreaks
    \begin{align}
        p^{(m)}_{(y_1,y_2)} &\leq (1 - \delta)^{-2m}\left(\frac{(d-m)!}{d!}\right)^2e^{-2 + 10\delta} \sum_{a_1 + b_1 = m,\atop a_2 + b_2 = m} \sum_{V_1 \subseteq X_1, \atop |V_1| = a_1}\sum_{V_2 \subseteq X_2, \atop |V_2| = a_2}\sum_{U_1 \subseteq Z, \atop |U_1| = b_1}\sum_{U_2 \subseteq Z\setminus U_1, \atop |U_2| = b_2} 1 \nonumber\\
        %------------------
        &\leq (1 - \delta)^{-2m}\left(\frac{(d-m)!}{d!}\right)^2e^{-2 + 10\delta}\binom{d_H(y_1)}{m} \binom{d_H(y_2)}{m} \nonumber\\
        %------------------
        &\leq (1 - \delta)^{-2m} (1 + \delta)^{4m} e^{10\delta} \left( \frac{e^{-1}}{m!} \right)^2 \label{eq:middle-reuse}\\
        %------------------
        &\leq (1 + 100(k+1)\delta)\left( \frac{e^{-1}}{m!} \right)^2. \nonumber 
    \end{align}
    The inequality \eqref{eq:middle-reuse} holds by \eqref{eq:middle}.
    Now, from \eqref{eq:y1y2-linearlity}, we obtain the following.

    \begin{equation}\label{eq:second-moment}
        \mathbb{E}[I_m^2] \leq (1 + 100(k+1)\delta)\left( \frac{e^{-1}}{m!}n \right)^2.
    \end{equation}

    Since $\mathbb{E}[I_m]^2 \geq (1 - 100(k+1)\delta)^2 \left(\frac{e^{-1}}{m!}n \right)^2$ by Lemma~\ref{lem:expectation}, we obtain the following inequality from \eqref{eq:second-moment}.

    \allowdisplaybreaks
    \begin{align*}
        \mathrm{Var}[I_m] &= \mathbb{E}[I_m^2] - \mathbb{E}[I_m]^2\\
        %---------------------
        &\leq \left((1 + 100(k+1)\delta) - (1 - 100(k+1)\delta)^2\right)\left(\frac{e^{-1}}{m!}n \right)^2\\
        %---------------------
        &\leq 10000(k+1)\delta \left(\frac{e^{-1}}{m!}n \right)^2.
    \end{align*}
    This completes the proof.
\end{proof}

Up to this point, we have established bounds on the expectation and variance of the random variables $I_j$ for all $0 \le j \le k$. As a consequence, we can now estimate the median of the random variable $I_{>m}$ as follows.

\begin{lemma}\label{lem:median-naive-concentration}
    Let $k, \delta, d_0, d, X, Y, H$ be the real numbers, vertex sets, and the bipartite graph as in the statement of Theorem~\ref{thm:balls-into-bins-regular}. Then for each $0 \leq m \leq k$, we have
    $$
        \Pr\left[\left| I_{>m} - \sum_{i > m} \frac{e^{-1}}{i!}n\right| \leq 300(k+1)\sqrt{\delta}\cdot n  \right] \geq \frac{1}{2}.
    $$
\end{lemma}

\begin{proof}[Proof of Lemma~\ref{lem:median-naive-concentration}]
    By Chebyshev's inequality and Lemma~\ref{lem:variance}, for each $0 \leq j \leq m$, we have
    \begin{equation}\label{eq:chebyshev}
        \Pr\left[ \left| I_j - \mathbb{E}[I_j] \right| > 200(k+1)\sqrt{\delta} \cdot \frac{e^{-1}}{j!}n \right] \leq \frac{1}{2(k+1)}.
    \end{equation}
    %---------------
    Together with \eqref{eq:chebyshev}, Lemma~\ref{lem:expectation}, and the triangle inequality, the following holds.
    \begin{equation}\label{eq:chebyshev}
        \Pr\left[ \left| I_j - \frac{e^{-1}}{j!}n \right| > 300(k+1)\sqrt{\delta} \cdot \frac{e^{-1}}{j!}n \right] \leq \frac{1}{2(k+1)}.
    \end{equation}
    %---------------
    Thus, we obtain the following inequality by the union bound.
    \allowdisplaybreaks
    \begin{align*}
        &\Pr\left[ \left| \sum_{0\leq j \leq m} I_j - \sum_{0 \leq j \leq m} \frac{e^{-1}}{j!}n \right| > 300(k+1)\sqrt{\delta} \cdot n \right]\\
        %---------------
        &\leq \Pr\left[ \left| \sum_{0\leq j \leq m} I_j - \sum_{0 \leq j \leq m} \frac{e^{-1}}{j!}n \right| > 300(k+1)\sqrt{\delta} \cdot \left(\sum_{0\leq j \leq m} \frac{e^{-1}}{j!}n\right) \right]\\
        &\leq \sum_{0\leq j \leq m} \Pr\left[ \left| I_j - \frac{e^{-1}}{j!}n \right| > 300(k+1)\sqrt{\delta} \cdot \frac{e^{-1}}{j!}n \right]\\
        %----------------
        &\leq \frac{m+1}{2(k+1)}\\
        %----------------
        &\leq \frac{1}{2}.
    \end{align*}

    Since $I_{>m} = n - \sum_{0\leq j \leq m} I_j$ and $\sum_{i>m}\frac{e^{-1}}{i!}n = n - \sum_{0\leq j \leq m}\frac{e^{-1}}{j!}n$, we obtain the desired inequality
    \begin{equation*}
        \Pr\left[ \left| I_{>m} - \sum_{i > m} \frac{e^{-1}}{i!}n \right| \leq 300(k+1)\sqrt{\delta}\cdot n \right] \geq \frac{1}{2}.
    \end{equation*}
    This completes the proof.
\end{proof}

We are now ready to prove Theorem~\ref{thm:balls-into-bins-regular}.

\begin{proof}[Proof of Theorem~\ref{thm:balls-into-bins-regular}]
    Let $0 \leq m \leq k$.
    Our strategy is to apply Lemma~\ref{lem:Talagrand} on the random variable $I_{> m}$. To achieve this, we need to check the Lipschitz and certificate conditions of $I_{>m}$. Observe that if some vertex $x_i \in X$ changes its choice among the vertices of $N_H(x_i)$, then $I_{>m}$ differs by at most $1$. This means that $I_{>m}$ is a $1$-Lipschitz function. Now assume that $I_{>m}(\mathbf{y}) \geq \ell$. Let $Y'\subseteq Y$ be a vertex subset of $Y$ of size $\ell$ such that each vertex in $Y'$ has degree at least $m+1$ in the output graph $H'_{\mathbf{y}}$. For each $y'\in Y'$, there are $m+1$ vertices in $X$ such that they chose $y'$ in the random process on $(X, Y; H)$-model. Now consider the union of such $m+1$ vertices for all $y'\in Y'$, say $X'$. Then it means once the choices of $X'$ are fixed as the same as $\mathbf{y}$, then whatever the other choices were made, the value of $I_{>m}$ is at least $\ell$. Also, since $|Y'| = \ell$, and each $y' \in Y'$ contributes at most $m+1$ vertices in $X'$, we have $|X'|\leq (m+1)\ell$. This means that the function $f_m: x\to (m+1)x$ is the certificate of $I_{>m}$. Hence, by Lemma~\ref{lem:Talagrand}, we have
    
    \begin{align}
        \Pr[I_{>m} \leq b - t\sqrt{(m+1)b}] \Pr[I_{>m} \geq b] &\leq e^{-t^2/4} \label{eq:talagrand-lower}\\
        \Pr[I_{>m} \leq b'] \Pr[I_{>m} \geq b' + t'\sqrt{(m+1)b'}] &\leq e^{-t'^2/4}. \label{eq:talagrand-upper}
    \end{align}
    for any $b, b'$ and $t, t' \geq 0$.

    Choose $b = \sum_{i>m} \frac{e^{-1}}{i!}n - 300(k+1)\sqrt{\delta} \cdot n$ and $t = \frac{\gamma/2 - 300(k+1)\sqrt{\delta}}{\sqrt{(m+1)b}}n$. Note that since $\delta \leq \frac{10^{-8}}{(k+2)!^2}$, we have $b \geq \frac{1}{2}\sum_{i > m}\frac{e^{-1}}{i!}n$. Thus, by Lemma~\ref{lem:median-naive-concentration} and \eqref{eq:talagrand-lower}, the following holds.

    \allowdisplaybreaks
    \begin{align*}
        &\frac{1}{2}\Pr\left[I_{>m} \leq \sum_{i>m} \frac{e^{-1}}{i!}n - \frac{\gamma}{2} n\right]\\
        %------------------------
        &= \frac{1}{2}\Pr[I_{>m} \leq b - (\gamma/2 - 300(k+1)\sqrt{\delta})n]\\
        %------------------------
        &\leq \Pr[I_{>m} \leq b - (\gamma/2 -300(k+1)\sqrt{\delta})n]\Pr[I_{>m} \geq b]\\
        %------------------------
        &\leq e^{- \frac{(\gamma/2 - 300(k+1)\sqrt{\delta})^2n^2}{4(m+1)b}}\\
        %------------------------
        &\leq e^{-\frac{(\gamma - 600(k+1)\sqrt{\delta})^2}{16(k+1)}n}.
    \end{align*}
    Hence, we have
    \begin{equation}\label{eq:lower-I>m}
        \Pr\left[I_{>m} \leq \sum_{i>m} \frac{e^{-1}}{i!}n - \frac{\gamma}{2} n\right] \leq 2 e^{-\frac{(\gamma - 600(k+1)\sqrt{\delta})^2}{16(k+1)}n}. 
    \end{equation}

    Similarly, choose $b' = \sum_{i>m} \frac{e^{-1}}{i!}n + 300(k+1)\sqrt{\delta} \cdot n$ and $t' = \frac{\gamma/2 - 300(k+1)\sqrt{\delta}}{\sqrt{(m+1)b}}n$. Note that $b' \leq 2n$ by our choice of $\delta$. Then by Lemma~\ref{lem:median-naive-concentration} and \eqref{eq:talagrand-upper}, we have
    
    \allowdisplaybreaks
    \begin{align*}
        &\frac{1}{2}\Pr\left[I_{>m} \geq \sum_{i>m} \frac{e^{-1}}{i!}n + \frac{\gamma}{2} n\right]\\
        %------------------------
        &= \frac{1}{2}\Pr[I_{>m} \leq b' + (\gamma/2 - 300(k+1)\sqrt{\delta})n]\\
        %------------------------
        &\leq \Pr[I_{>m} \leq b'] \Pr[I_{>m} \geq b' + (\gamma/2 -300(k+1)\sqrt{\delta})n]\\
        %------------------------
        &\leq e^{- \frac{(\gamma/2 - 300(k+1)\sqrt{\delta})^2n^2}{4(m+1)b}}\\
        %------------------------
        &\leq e^{-\frac{(\gamma - 600(k+1)\sqrt{\delta})^2}{32(k+1)}n}.
    \end{align*}
    Hence, we have
    \begin{equation}\label{eq:upper-I>m}
        \Pr\left[I_{>m} \leq \sum_{i>m} \frac{e^{-1}}{i!}n - \frac{\gamma}{2} n\right] \leq 2 e^{-\frac{(\gamma - 600(k+1)\sqrt{\delta})^2}{32(k+1)}n}. 
    \end{equation}

    From \eqref{eq:lower-I>m} and \eqref{eq:upper-I>m} with the union bound, we obtain the following sharp concentration inequality of the random variable $I_{>m}$.    

    \begin{equation}\label{eq:concentration-I>m}
        \Pr\left[ \left| I_{> m} - \sum_{i > m} \frac{e^{-1}}{i!}n \right| > \frac{\gamma}{2} n \right] \leq 4e^{-\frac{(\gamma - 600(k+1)\sqrt{\delta})^2}{32(k+1)}n}. 
    \end{equation}

    We now consider the case $k = 0$. Then $I_0 = n - I_{>0}$. Then by \eqref{eq:concentration-I>m}, we have
    \begin{equation*}
        \Pr\left[ \left| I_0 - \frac{1}{e}n \right| > \gamma n \right] \leq \Pr\left[ \left| I_{> 0} - \sum_{i > 0} \frac{e^{-1}}{i!}n \right| > \frac{\gamma}{2} n \right] \leq 4e^{-\frac{(\gamma - 600(k+1)\sqrt{\delta})^2}{32(k+1)}n}, 
    \end{equation*}
    which satisfies the inequality \eqref{eq:exponential-decay}.
    Thus, we may assume $k > 0$. Since $I_k = I_{> k-1} - I_{k}$, by \eqref{eq:concentration-I>m} with the union bound, the following holds.
    \allowdisplaybreaks
    \begin{align*}
        &\Pr\left[ \left| I_k - \frac{e^{-1}}{k!}n \right| > \gamma n \right]\\
        %-------------------
        &\leq \Pr\left[ \left| I_{> k-1} - \sum_{i > k-1} \frac{e^{-1}}{i!}n \right| > \frac{\gamma}{2}n \right] + \Pr\left[ \left| I_{> k} - \sum_{i > k} \frac{e^{-1}}{i!}n \right| > \frac{\gamma}{2}n \right]\\
        %-------------------
        &\leq 10e^{-\frac{(\gamma - 600(k+1)\sqrt{\delta})^2}{32(k+1)}n}.
    \end{align*}
    This completes the proof.
\end{proof}

%-------------------------------------------------------

\section{Proof of Theorem~\ref{thm:main-anticoncentration}}\label{sec:proof}

In this section, we prove Theorem~\ref{thm:main-anticoncentration} by showing the following. For a given $n$-vertex graph $G$ and an $n$-vertex tree $T$, denote by $N(T; G)$ the number of labeled copies of $T$ in $G$. Note that the ratio between the number of isomorphic copies and the labeled copies of $T$ in $G$ is exactly the number of automorphisms of $T$.

\begin{theorem}\label{thm:counting}
    There exists a real number $d'_0 > 0$ such that for all $d \geq d'_0$, the following holds. Let $G$ be an $n$-vertex $(1 \pm 10^{-10})d$-regular graph and $T$ be an $n$-vertex tree. Then we have
    $$
        N(T; G) \leq e^{-n/1000}d^n.
    $$
\end{theorem}

Keep Theorem~\ref{thm:counting} in our mind, we can prove Theorem~\ref{thm:main-anticoncentration}.

\begin{proof}[Proof of Theorem~\ref{thm:main-anticoncentration}]
    Let $d'_0$ be the same real number as in the statement of Theorem~\ref{thm:counting}.
    By Theorem~\ref{thm:Kostochka}, there exists a constant $d''_0 > 0$ such that every connected $n$-vertex $(1\pm 10^{-10})d$-regular graph $G$ satisfies the inequality
    \begin{equation}\label{eq:tree-lowerbound}
        (1 - 2\cdot 10^{-10})^n d^n \leq |T(G)|
    \end{equation}
    whenever $d \geq d''_0$.
    Let $d_0 \defeq \max\{d'_0, d''_0\}$. Fix an integer $d \geq d_0$ and consider a connected $n$-vertex $(1 \pm 10^{-10})d$-regular graph $G$. Let $T$ be an arbitrary $n$-vertex tree. Let $\calT$ be a random spanning tree of $G$ that is chosen uniformly at random from the set $T(G)$. Then by Theorem~\ref{thm:counting} and \eqref{eq:tree-lowerbound}, we have
    $$
        \Pr[\mathcal{T} \simeq_{\mathrm{iso}} T] \leq \frac{e^{-n/1000}d^n}{|T(G)|} \leq \frac{e^{-n/1000}d^n}{(1 - 2\cdot 10^{-10})^n d^n} \leq e^{-n/2000}.
    $$
    This completes the proof.
\end{proof}

Hence, the only remaining task is to prove Theorem~\ref{thm:counting}. To prove Theorem~\ref{thm:counting}, we need Br\'{e}gman--Minc inequality~\cite{Bregman}. For a given graph $G$, we denote $\mathrm{pm}(G)$ as the number of perfect matchings in $G$.

\begin{lemma}[Br\'{e}gman--Minc inequality]\label{lem:Bregman-Minc}
    Let $H$ be a bipartite graph on the vertex partition $V(H) = A \cup B$ with $|A| = |B|$. Then we have
    \begin{equation*}
        \mathrm{pm}(H) \leq \prod_{v\in A} (d_H(v)!)^{1/d_H(v)}.
    \end{equation*}
\end{lemma}

We now prove Theorem~\ref{thm:counting}.

\begin{proof}[Proof of Theorem~\ref{thm:counting}]
    Let $d_1$ denote the real number corresponding to $d_0$ in the statement of Theorem~\ref{thm:balls-into-bins-regular} with parameters $k = 0$ and $\delta = 10^{-10}$.
    Furthermore, let $d_2$ be the smallest real number such that, for all $d \geq d_2$,
    \begin{equation}\label{eq:stirling}
        \left((1+10^{-10})d\right)^{1/((1+10^{-10})d)} \leq (1+10^{-9})\frac{d}{e}.
    \end{equation}
    Note that $d_2$ exists from the fact that $x \to x^x$ is a monotonically increasing function in the interval $x\in [1, \infty)$ and also by Stirling's approximation. Set $d'_0 = \max\{d_1, d_2\}$, fix an $n$-vertex $(1 \pm 10^{-10})d$-regular graph $G$ and $n$-vertex tree $T$, where $d \geq d_0$. Let $D$ be a digraph which is obtained from $G$ by replacing every edge of $G$ by a digon.

    \begin{claim}\label{clm:few-leaves}
        If $T$ has less than $\frac{n}{10}$ leaves, then $N(T; G) \leq e^{-n/1000}d^n$.
    \end{claim}

    \begin{claimproof}[Proof of Claim~\ref{clm:few-leaves}]
        Let $X$ and $Y$ be two disjoint copies of $V(G)$. We construct a bipartite graph $H$ on the vertex partition $X \cup Y$ such that $(u,v)\in X\times Y$ is an edge of $H$ if and only if $uv\in E(G)$. Observe that $H$ is also $(1 \pm 10^{-10})d$-regular. Consider a $(X, Y; H)$-model and its output graph $H'$. Based on $H'$, we define $D_{H'}$ to be the sub-digraph of $D$ in which every vertex $u \in V(G)$ has a unique out-neighbor, namely the vertex $v \in V(G)$ such that $(u,v) \in H'$. Note that the mapping $H' \to D_{H'}$ is a bijection between the set of all output graphs of $(X, Y; H)$-model and the set of all spanning sub-digraphs of $D$ with all the vertices having out-degree exactly $1$. Since $H'$ is $(1\pm 10^{-10})d$-regular, the size of such sets are at most $(1 + 10^{-10})^n d^n$.

        Now we fix an arbitrary edge $ab\in E(T)$ in $T$, which is not a leaf edge, and replace it with a digon. Then we replace the edges of subtrees of $T$ rooted at $a$ and rooted at $b$ by a directed edge with direction toward the root, respectively. We denote by $\overrightarrow{T_{ab}}$ the directed tree obtained from this, whose underlying graph is isomorphic to $T$. Note that every vertex of $\overrightarrow{T_{ab}}$ has out-degree exactly one in $\overrightarrow{T}_{ab}$ and each leaf vertex of $T$ has in-degree zero in $\overrightarrow{T_{ab}}$. Since $T$ has at most $\frac{n}{10}$ leaves, the number of labeled copies of $T$ in $G$ is bounded above by the number of spanning sub-digraphs of $D_H$ such that all the vertices have out-degree exactly one and contains at most $\frac{n}{10}$ vertices that have in-degree zero. Thus, by Theorem~\ref{thm:balls-into-bins-regular}, we have
        \begin{equation*}
            N(T;G) \leq \Pr\left[I_0 \leq \frac{n}{10} \right] (1 + 10^{-10})^n d^n \leq e^{-n/1000}d^n.
        \end{equation*}
        This completes the proof.
    \end{claimproof}

    We now consider the case that $T$ has many leaves.

    \begin{claim}\label{clm:many-leaves}
        If $T$ has at least $\frac{n}{10}$ leaves, then $N(T; G) \leq e^{-n/1000}d^n$.
    \end{claim}

    \begin{claimproof}[Proof of Claim~\ref{clm:many-leaves}]
        Fix an arbitrary non-leaf vertex $v$ of $T$.
        Consider the directed tree obtained from $T$ by orienting every edge away from $v$; equivalently, consider the out-branching tree of $T$ rooted at $v$. We denote this directed tree by $\overrightarrow{T_v}$. 

        Let $L$ be the set of leaf vertices of $T$ and $R$ be the set of vertices of $T$ that are incident with at least one leaf vertex. We denote by $\overrightarrow{T'_v}$ the directed tree that is obtained from $\overrightarrow{T_v}$ by removing all vertices in $L$.

        Let $\mathrm{Iso}(\overrightarrow{T_v}; D)$ be the set of isomorphisms from $\overrightarrow{T_v}$ to $d$.
        Observe that the number of labeled copies of $\overrightarrow{T_v}$ in $D$ is exactly $N(T; G)$, which is bounded above by the number of isomorphisms from $\overrightarrow{T_v}$ to $D$. Thus, we have
        \begin{equation}\label{eq:bound-by-iso}
            N(T; G) \leq |\mathrm{Iso}(\overrightarrow{T_v}; D)|.
        \end{equation}
        
        Define the set of partial embeddings $\mathrm{Emb}(\overrightarrow{T'_v}, D)$ as below.
        $$
            \mathrm{Emb}(\overrightarrow{T'_v}, D) \defeq \{\phi: \exists \psi \in \mathrm{Iso}(\overrightarrow{T_v}; D) \text{ s.t. } \psi|_{V(\overrightarrow{T'_v})} = \phi\}.
        $$
        %--------------
        For each $\phi\in \mathrm{Emb}(\overrightarrow{T'_v}, D)$, we now consider the set
        $$
            \theta_{\phi} \defeq \{\psi \in \mathrm{Iso}(\overrightarrow{T_v};D): \psi|_{V(\overrightarrow{T'_v})}= \phi\}.
        $$
        %----------------
        From the definition of $\mathrm{Emb}(\overrightarrow{T'_v}, D)$, observe that 
        $$
            \mathrm{Iso}(\overrightarrow{T_v}; D) = \bigcup_{\phi\in \mathrm{Emb}(\overrightarrow{T'_v}; D)} \theta_{\phi}.  
        $$
        %------------
        Hence, by \eqref{eq:bound-by-iso}, we deduce the following.
        \begin{equation}
            N(T; G) \leq |\mathrm{Emb}(\overrightarrow{T'_v}; D)| \cdot \max\{|\theta_{\phi}| : \phi\in \mathrm{Emb}(\overrightarrow{T'_v}; D)\}.
        \end{equation}
        This implies that it suffices to bound the size of the sets $\mathrm{Emb}(\overrightarrow{T'_v}; D)$ and $\theta_{\phi}$.
        
        We consider a sequence of directed trees $\overrightarrow{T'_1}, \dots, \overrightarrow{T'_i}, \dots, \overrightarrow{T'_m}$, where $\overrightarrow{T'_1}$ consists of the single vertex $v$ and $\overrightarrow{T'_m} = \overrightarrow{T'_v}$. For each $i$, the directed tree $\overrightarrow{T'_{i+1}}$ is obtained from $\overrightarrow{T'_i}$ by choosing a vertex $u_i \in V(\overrightarrow{T'_i})$ for which there exists a vertex $w_i \in V(\overrightarrow{T'_v}) \setminus V(\overrightarrow{T'_i})$ such that $\overrightarrow{u_iw_i}$ is a directed edge of $\overrightarrow{T'_v}$, and then adding the vertex $w_i$ together with the directed edge $\overrightarrow{u_iw_i}$ to $\overrightarrow{T'_i}$. For each $i$ and a given embedding $\phi_i$ of $\overrightarrow{T'_i}$ in $D$, the number of embedding $\phi_{i+1}$ of $\overrightarrow{T'_{i+1}}$ in $D$ that extends $\phi_i$ is at most $d^+_D(\phi_i^{-1}(u_i)) = d_G(\phi_i^{-1}(u_i)) \leq (1 + 10^{-10})d$. Since $\overrightarrow{T'_0}$ is a single vertex, the number of different embeddings of $\overrightarrow{T'_0}$ is $n$. Hence, we have
        \begin{equation}\label{eq:emb-upper}
            |\mathrm{Emb}(\overrightarrow{T'_v}; D)| \leq n \cdot (1 + 10^{-10})^{n - 1 -|L|}d^{n - 1 - |L|}.
        \end{equation}

        Fix an arbitrary $\phi \in \mathrm{Emb}(\overrightarrow{T'_v}; D)$. Let $X \defeq V(D) \setminus \phi(V(\overrightarrow{T'_v}))$ and $Y \defeq \phi(R)$. We now consider a bipartite graph $H$ on the vertex bipartition $X \cup Y$ such that $xy\in X\times Y$ is an edge of $H$ if and only if $xy\in E(G)$. Then we define a bipartite graph $H'$ by blowing up each vertex $y\in Y$ by $L_y$ times, where $L_y$ is the number of leaves in $T$ that are incident with the vertex $\phi^{-1}(y)$. Denote $Y'$ as the vertex set obtained from blowing up each vertex $y\in Y$ by $L_y$ times. Then $H'$ is a bipartite graph on the vertex partition $X \cup Y'$ with $|X| = |Y'| = |L| \geq \frac{n}{10}$ since $|Y'| = \sum_{y\in Y} L_y$ and $\sum_{y\in Y} L_y$ is exactly the same number of leaves of $T$, which is $|L|$. 
        As $G$ is $(1 \pm 10^{-1})d$-regular, each $y'\in Y'$ has degree at most $(1 + 10^{-10})d$. Observe that each isomorphism from $\overrightarrow{T'_v}$ to $D$ that extends $\phi$ corresponds to each perfect matching in $H'$ since $L$ is the set of leaves of $T$. Thus, by Lemma~\ref{lem:Bregman-Minc} and our choice of $d_0$, we have
        \begin{equation}\label{eq:theta-phi}
            |\theta_{\phi}| = pm(H') \leq \left(((1 + 10^{-10})d)!^{1/((1 + 10^{-10})d)}\right)^{|L|} \leq (1 + 10^{-9})^{|L|}e^{-|L|}d^{|L|}.
        \end{equation}

        Hence, from \eqref{eq:bound-by-iso}, \eqref{eq:emb-upper}, and \eqref{eq:theta-phi}, we obtain the following.
        \begin{equation*}
            N(T; G) \leq (1 + 10^{-9})^{n - 1}e^{-|L|}d^{n - 1}.
        \end{equation*}
        As $|L| \geq \frac{n}{10}$, we finally deduce
        \begin{equation*}
            N(T; G) \leq e^{-n/1000}d^n.
        \end{equation*}
        This completes the proof.
    \end{claimproof}
    
    As a direct consequence of Claims~\ref{clm:few-leaves} and \ref{clm:many-leaves}, we obtain the desired inequality for any $n$-vertex tree $T$. This completes the proof.
\end{proof}

%-------------------------------------------------------

\section{Concluding remarks}\label{sec:concluding}

In this paper, we studied an anticoncentration phenomenon of random spanning trees in almost regular graphs, and as a corollary, we deduce that a connected almost regular graph with sufficiently high degree contains exponentially many pairwise non-isomorphic spanning trees. Around this theme, we have many open problems and interesting directions for further research.

%--------------------

\textbf{$\bullet$ Threshold for $d_0$.}
We wonder whether Theorem~\ref{thm:main-non-isomorphic} still holds for small degree cases. We believe one can take $d_0 = 3$.
\begin{conjecture}\label{conj:start-cubic}
    Let $n, d \geq 3$ be integers. Then there exists a universal constant $c > 0$ that only depends on $d$ such that the following holds for every connected $n$-vertex $d$-regular graph $G$.
    $$
        |T_{\mathrm{unlabeled}}(G)| \geq e^{cn}.
    $$
\end{conjecture}

%---------------------

\textbf{$\bullet$ Approaches to Otter's formula.}
Let $\alpha$ be the constant that appears in Otter's formula (Theorem~\ref{thm:otter}). Our main theorem, Theorem~\ref{thm:main-non-isomorphic}, states that one can find exponentially many non-isomorphic spanning trees in a connected almost regular graph, but the constant of the base is far from $\alpha$. It would be interesting to show that the base approaches $\alpha$ as $d$ goes to infinity. 

\begin{question}\label{ques:approach-to-alpha}
    For a given real number $\ve > 0$, is there a constant $d_{\ve} > 0$ such that the following holds for all $d \geq d_{\ve}$?
    For every connected $n$-vertex $d$-regular graph $G$, 
    $$
        |T_{\mathrm{unlabeled}}(G)| \geq (\alpha - \ve)^n.
    $$
\end{question}
By considering the random graph with not too small edge density, we believe that the optimal anticoncentration inequality would be in the form below.
$$
    \Pr\bigl[\mathcal{T} \simeq_{\mathrm{iso}} T\bigr] \leq e^{-(1-o(1))n}.
$$
Since $e < \alpha$, this means that to answer Question~\ref{ques:approach-to-alpha} affirmatively, it requires a different approach than what we used in this paper.

%-----------------------

\textbf{$\bullet$ Optimal anticoncentration inequality.}
As we already referred to in the previous paragraph, we expect that the optimal anticoncentration inequality would be attained from the random graph. Let $n$ be a sufficiently large number and $p > 0$ be a real number such that $p \gg \log n/ n$. For a given $n$-vertex tree $T$, the number of labelled copies of $T$ in $K_n$ is $\frac{n!}{|\mathrm{Aut}(T)|} = \left( \frac{(1 + o(1))n}{e} \right)^n \cdot \frac{1}{|\mathrm{Aut}(T)|}$, where $\mathrm{Aut}(T)$ is the automorphism group of $T$. Hence, the expected number of labeled copies of $T$ in $G_{n, p}$ is $\left( \frac{(1 + o(1))np}{e} \right)^n \cdot \frac{1}{|\mathrm{Aut}(T)|}$, and $G_{n, p}$ is almost regular graph with degrees close to $np$ with high probability. Hence, it is natural to conjecture the following.

\begin{conjecture}\label{conj:optimal-anticoncentration}
    For a given real number $\ve > 0$, there exists a constant $d_{\ve} > 0$ such that the following holds for all integers $d \geq d_{\ve}$. Let $G$ be a connected $n$-vertex $d$-regular graph and $\calT$ be a uniformly random spanning tree of $G$. Then for all $n$-vertex tree $T$, we have
    $$
        \Pr\bigl[\calT \simeq_{\mathrm{iso}} T\bigr] \leq e^{-(1-\ve)n}.
    $$
\end{conjecture}
As supporting information, we note that the Br\'{e}gman--Minc inequality implies that $n$-vertex path satisfies Conjecture~\ref{conj:optimal-anticoncentration}.
We also remark that the bound in Conjecture~\ref{conj:optimal-anticoncentration} is best possible up to lower-order terms. Indeed, one can easily construct a $d$-regular graph that contains $(d/e)^{(1-o_d(1))n}$ distinct $n$-vertex paths, therefore, we omit the details. It would also be very interesting to determine whether Conjecture~\ref{conj:optimal-anticoncentration} holds for an $(n, d, \lambda)$-graph with appropriate parameters $d$ and $\lambda$.

%----------------------

\textbf{$\bullet$ Graphs with minimum degree condition.}
We now consider replacing the almost regularity assumption with a minimum degree condition. As noted in Section~\ref{sec:intro}, Theorems~\ref{thm:main-non-isomorphic} and~\ref{thm:main-anticoncentration} no longer hold under the weaker assumption that the minimum degree is at least \(d\), with \(K_{d,n-d}\) serving as a counterexample. However, even the graph $K_{d, n-d}$ contains at least polynomially many pairwise non-isomorphic spanning trees. Thus, we conjecture that a polynomial anticoncentration phenomenon still exists even under the minimum degree assumption, which is essentially the weakest assumption.

\begin{conjecture}\label{conj:min-degree-polynomial}
    There exist universal constants $d_0, c > 0$ such that the following holds for all integers $d \geq d_0$. Let $G$ be a connected $n$-vertex graph with minimum degree at least $d$ and $\calT$ be a uniformly random spanning tree of $G$. Then for all $n$-vertex tree $T$, we have
    $$
        \Pr\bigl[\calT \simeq_{\mathrm{iso}} T\bigr] \leq n^{-c}.
    $$
\end{conjecture}

%-----------------------

\textbf{$\bullet$ Other graph families.} Let $\calF$ be a collection of unlabelled graphs and $G$ be a host graph. Denote $\calF(G)$ as the set of labeled copies of members in $\calF$ that appear as subgraphs of $G$. We now consider the quantity
$$
    q(\calF; G) \defeq \max\{\Pr[\calF' \simeq_{\mathrm{iso}} F]: F\in \calF\},
$$
where $\calF'$ is a graph chosen uniformly at random from $\calF(G)$. With this notation, Theorem~\ref{thm:main-anticoncentration} can then be rephrased as asserting that $q(\calF; G)$ is exponentially small if $\calF$ is the set of all $n$-vertex trees and $G$ is a connected $n$-vertex almost regular graph. Then the natural question is to investigate the quantity $q(\calF; G)$ for other graph families $\calF$ and graphs $G$, for instance, one can choose $\calF$ as an $n$-vertex $2k$-regular graph and $G$ as an $n$-vertex $2\ell$-regular graph, where $1 \leq k < \ell$. We also believe that considering hypergraphs and digraphs would be very interesting.

%-----------------------

\textbf{$\bullet$ Further research on $(X, Y; H)$-model.} Finally, we comment on the $(X, Y; H)$-model introduced in Section~\ref{sec:balls-into-bins}. Since applications of the classical balls--into--bins model have been extensively studied, we expect that its graph-theoretic generalization, the $(X, Y; H)$-model, will admit a variety of further applications. Exploring this model for various graphs $H$ would be an interesting direction for future work.

%-------------------------------------------------------
\section*{Acknowledgement}
The author thanks Sang-il Oum for helpful comments on the introduction and references.

%--------------------------------------------------------

\vspace{-0.2cm}

\providecommand{\MR}[1]{}
\providecommand{\MRhref}[2]{%
  \href{http://www.ams.org/mathscinet-getitem?mr=#1}{#2}
}

    \bibliographystyle{amsplain_initials_nobysame}
    \bibliography{bibfile}

%--------------------------------------------------------

\end{document}